\documentclass[12pt,a4paper]{amsart}

\usepackage{amsmath,amsfonts,amssymb,mathtools,extarrows}

\usepackage[hmargin=2.5cm,vmargin=2.5cm]{geometry}
\usepackage[dvipsnames]{xcolor}	
\usepackage{hyperref}
\usepackage{nameref,zref-xr}  
\usepackage{mathrsfs}
\usepackage{tikz-cd}
\usepackage{mathbbol}

\hypersetup{
colorlinks=true, linktocpage=true, pdfstartpage=1, pdfstartview=FitV,breaklinks=true, pdfpagemode=UseNone, pageanchor=true, pdfpagemode=UseOutlines,plainpages=false, bookmarksnumbered, bookmarksopen=true, bookmarksopenlevel=1,hypertexnames=true, pdfhighlight=/O,urlcolor=brown, linkcolor=RoyalBlue, citecolor=ForestGreen}

\usepackage{enumerate}      

\usepackage{scalerel}         
\usepackage{stmaryrd}

\usepackage[all]{xy}

\usepackage[textsize=footnotesize,textwidth=0.85in]{todonotes}
\presetkeys%
    {todonotes}%
    {color=Apricot}{}%
\newcommand{\oM}{\overline{\mathcal M}}

\def\oM{{\overline{\mathcal{M}}}}

\newcommand{\DR}{\mathrm{DR}}

\newcommand{\SRT}{\mathrm{SRT}}

\newcommand{\ZZ}{\mathbb{Z}}
\newcommand{\QQ}{\mathbb{Q}}

\newcommand{\Ch}{\Omega}

\newcommand{\hhh}{\text{{\Large {\<h>}}}}

\newtheorem{theorem}{Theorem}[section]

\newtheorem{lemma}[theorem]{Lemma}

\theoremstyle{remark}

\newtheorem{remark}[theorem]{Remark}

\theoremstyle{definition}

\newtheorem{definition}[theorem]{Definition}

\usepackage{color}

\newcommand{\bOm}{{\mathbb{\Pi}}}
\newcommand{\bOmT}{{\hhh \hspace{-7pt} \hhh}}
\newcommand{\OmA}{{\Omega\hspace{-7pt}\Omega\hspace{-5pt}A}}
\newcommand{\bUp}{{\Upsilon\hspace{-7pt}\Upsilon}}
\newcommand{\bD}{{\mathbb{D}}}

\usepackage{cjhebrew}
\usepackage{pdftexcmds}

\DeclareFontFamily{U}{cbgreek}{}
\DeclareFontShape{U}{cbgreek}{m}{n}{
        <-6>    grmn0500
        <6-7>   grmn0600
        <7-8>   grmn0700
        <8-9>   grmn0800
        <9-10>  grmn0900
        <10-12> grmn1000
        <12-17> grmn1200
        <17->   grmn1728
      }{}
\DeclareFontShape{U}{cbgreek}{bx}{n}{
        <-6>    grxn0500
        <6-7>   grxn0600
        <7-8>   grxn0700
        <8-9>   grxn0800
        <9-10>  grxn0900
        <10-12> grxn1000
        <12-17> grxn1200
        <17->   grxn1728
      }{}

\makeatletter
\newcommand{\normalorbold}{%
  \ifnum\pdf@strcmp{\math@version}{bold}=\z@ bx\else m\fi
}
\makeatother

\numberwithin{equation}{section}

\begin{document}

\title[$\Omega$-classes and double ramification cycles]{Rooted trees with level structures, $\Omega$-classes and double ramification cycles}

\author[X.~Blot]{Xavier Blot}
\address{X.~B.: Korteweg-de Vriesinstituut voor Wiskunde, Universiteit van Amsterdam, Postbus 94248, 1090GE Amsterdam, Nederland}
\email{x.j.c.v.blot@uva.nl}	

\author[D.~Lewa\'nski]{Danilo Lewa\'nski}
\address{D.~L.: Dipartimento di Matematica, Informatica e Geoscienze, Universit\`a degli studi di Trieste,
	Via Weiss 2, 34128	Trieste, Italia}
\email{danilo.lewanski@units.it} 

\author[S.~Shadrin]{Sergey Shadrin}
\address{S.~S.: Korteweg-de Vriesinstituut voor Wiskunde, Universiteit van Amsterdam, Postbus 94248, 1090GE Amsterdam, Nederland}
\email{s.shadrin@uva.nl}	

\begin{abstract} We prove a new system of relations in the tautological ring of the moduli space of curves involving stable rooted trees with level structure decorated by the top Chern class of the Hodge bundle and $\Omega$-classes and double ramification structures. In particular, this resolves a recent conjecture on these relations as well as connects with one of the two sides of the recently established DR/DZ equivalence between the integrable hierarchies constructions of Buryak and of Dubrovin--Zhang. 
\end{abstract}

\maketitle
\tableofcontents

\section{Introduction}

The goal of this work is to establish a conjecture previously proposed in \cite{BLRS}, which expresses the so-called $A$-classes emerging in the theory of the DR/DZ (double ramification / Dubrovin-Zhang) correspondence for integrable hierarchies \cite{BDGR1,BGR19} \textemdash \ in terms of sums of stable trees decorated by the products of the so-called $\lambda_g$-class and the so-called $\Omega$-class. The ideas and techniques introduced in this paper were recently used to establish the strong DR/DZ conjecture \cite{BLS}.\\

Let us recall that the motivation behind the original conjecture in \cite{BLRS} arises from a wider picture revolving around the DR/DZ correspondence and its quantisation.

The DR side of the equivalence stars double ramification (DR) hierarchies of PDEs associated to cohomological field theories, introduced by Buryak in~\cite{Bur} by means of the double ramification cycle and extended to the case of F-cohomological field theories in \cite{BR21,ABLR21}.

The DZ side stars Dubrovin-Zhang hierarchies \cite{DZ,BPS1,BPS2} of Hamiltonian equations. The DZ side historically arises earlier, although its construction does not generalise to non-semisimple CohFTs (though the results of~\cite{BLS,BS22} imply the construction for possibly non-semisimple tautological CohFTs) and no quantisation procedure is so far known, whereas both are indeed the case for the double ramification hierarchy. 

The two constructions were conjectured to be Miura equivalent, see~\cite{Bur}, in the sense of an explicit polynomial transformation of the variables bringing the equations from one side onto the other side of the correspondence. 
This correspondence was then refined as a relation between tau functions of both sides of the correspondence, called the strong DR/DZ conjecture, see \cite{BDGR1}. This relation of tau functions was finally lifted to a tautological relation between the $A$-class (DR side)  and the $B$-class (DZ side) in \cite{BGR19}, and generalized in \cite{BS22} to involve a larger class of F-cohomological field theories (as well as polynomiality of conservation laws of the latter). This final tautological relation was proved in the Gorenstein quotient of the moduli space of curves in \cite{BLS}, which is sufficient the prove the correspondence for tautological F-CohFT and the polynomiality of the DZ hierarchies in that case.

As mentioned above, another meaningful advantage of the Buryak's double ramification hierarchy is the construction of its quantisation, introduced in~\cite{BR16}, and of their quantum tau functions, introduced in~\cite{BDGR2}. These quantum tau functions were first studied in~\cite{Blot}, and conjectured to be represented by intersection numbers involving the $\Omega$-class and the $\lambda$-class, see~\cite{BlotLew}. In particular, this gives a new interpretation of tau functions of the DR hierarchies in the classical setting. This led to the formulation of a new conjecture in \cite{BLRS} relating the $A$-class with this time a class, denoted by $\bOmT$, defined as a sum stable trees decorated by products of the $\Omega$-class and the $\lambda_g$-class. 
\\

To summarize, we have three classes involved and they are conjectured to be equal: the $A$-class (DR side), the $B$-class (DZ side) and the $\bOmT$-class which is a natural class to represent the tau functions from the point of view of quantization. In addition, the $B$-class and the $\bOmT$-class are mirror to each other in the sense that they are both defined by the same stable trees (namely rooted stable trees with level structures), but the first one is decorated with $\psi$-classes rather the second one with products of the $\Omega$-class with the $\lambda_g$-class.

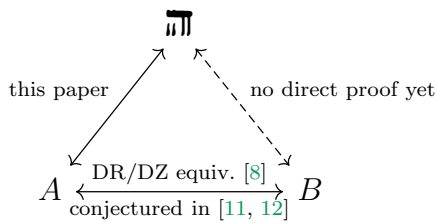
\[\begin{tikzcd}
& |[alias=C]| \bOmT \arrow[leftrightarrow, dl,"\text{this paper}"'] \arrow[leftrightarrow, dr, dashed, "\text{no direct proof yet}"]  &  \\[2em]
A  \arrow[leftrightarrow, rr, "\text{DR/DZ equiv.~\cite{BLS}}", "\text{conjectured in \cite{Bur, BDGR1}}"'] & & B
\end{tikzcd}\]
We prove in this paper the correspondence between the $A$-class and the $\bOmT$-class. By mirroring our approach we prove in \cite{BLS} the correspondence $A=B$ in the Gorenstein quotient. More precisely, the proof of this paper involves two tools:

\begin{enumerate}
    \item a virtual localization formula giving a tautological relation between DR cycles and $\Omega$-classes,
    \item a combinatorial argument involving rooted trees with level  structures.
\end{enumerate}
Since the $B$-class and the $\bOmT$-class are mirror of each other, it is sufficient to replace the role of the $\Omega$-class in the relation of (1) by $\psi$-classes to obtain a proof of $A=B$. In \cite{BLS}, we call this relation the master relation and prove it in the Gorenstein quotient. This results in a proof of $A=B$ in the Gorenstein quotient. The link between the $B$-class and the $\bOmT$-class is therefore also indirectly  established in the Gorenstein quotient through the DR/DZ correspondence, although no direct proof is as of yet available, despite being desirable.

As remarked in \cite{BLRS}, there is a \textit{per sè} interest in the involvement of the $\Omega$-classes within the framework of the DR/DZ equivalence, as they are by now playing a key role in more than twenty enumerative problems arising from algebraic geometry, integrability, random matrix models and string theory. See for instance \cite{rELSV, JPPZ, BKL, spin, Euler, Double} and references therein.

\subsection{Organization of the paper} The paper is not entirely self-contained, as we borrow several constructions and ideas from a number of previous works on similar structures, see~\cite{BLS} and~\cite{BGR19,BHS22, AS09, BS22}, and especially from the paper in which the conjecture this paper solves was formulated~\cite{BLRS}; the material presented in the last paper is necessary to follow our arguments. We assume the reader to be familiar with the standard notation for various classes in the tautological ring of the moduli spaces of curves and the usage of stable graphs for their expressions.

The introduction to the $\Omega$-classes and their properties is provided in Section~\ref{sec:Omega:classes}, where we recall their basic properties as well as the specialisation of the $\Omega$-class to what we call $\hhh$-class. 
We present our results in Section~\ref{sec:AOmega} assuming the reader is familiar with the standard stratification of the moduli spaces of curves and $\psi$-, $\kappa$-, and $\lambda$-classes in the tautological ring.
In Section~\ref{sec:virloc} we employ vitual localization techniques  in order to obtain linear systems of tautological relations which are employed in Section~\ref{sec:mainthm} to prove the main theorem of the paper. 


\subsection{Conventions and notation} We work with classes in the tautological ring of the moduli space of stable curves $R^*(\oM_{g,n})$; by ``cohomological degree'' by slight abuse of terminology we refer to the degree in the Chow ring, i.e. half of the actual degree in cohomology.  For a tautological class $c$ of mixed degree we denote by $(c)_i$ its homogenenous component of degree $i$ in Chow and $2i$ in cohomology.

\subsection{Acknowledgments} X.~B. and S.~S. are supported by the Netherlands Organization for Scientific Research. D.~L. is supported by the University of Trieste, by the INFN under the national project MMNLP, by the INdAM group GNSAGA, and thanks the University of Amsterdam for the hospitality during the research visit in which this project was developed.

\section{\texorpdfstring{$\Omega$}{Omega}-classes and their properties} \label{sec:Omega:classes}

\subsection{Definition and formula}
In 1983 Mumford \cite{Mum83} derived a formula for the Chern character of the Hodge bundle on the moduli space of curves $\overline{\mathcal{M}}_{g,n}$ in terms of tautological classes and Bernoulli numbers. Such class appears for instance in the ELSV formula \cite{ELSV01}, relating its integrals against monomials of $\psi$-classes with simple Hurwitz numbers.
%
%
A generalisation of Mumford's formula introducing the $s$-parameter was derived in 2005 by Bini~\cite{Bini05}. The generalisation we are going to employ introduces further $r$-th roots of a universal line bundle and was computed in 2008 by Chiodo~\cite{Chi08}. The moduli space $\overline{\mathcal{M}}_{g,n}$ is substituted by the proper moduli stack of generalized $r$-spin structures $\overline{\mathcal{M}}_{g}^{r,s}(a_1,\dots,a_n)$ which parametrizes $r$-th roots of the line bundle
\begin{equation*}
	\omega_{\log}^{\otimes s}\biggl(-\sum_{i=1}^n a_i p_i \biggr),
 \end{equation*}
where $\omega_{\log} = \omega(\sum_i p_i)$ is the log-canonical bundle, $r$ and $s$ are integers with $r$ positive, and $a_1,\dots,a_n$ are integers 
satisfying the modular constraint
\begin{equation*}
	a_1 + a_2 + \cdots + a_n \equiv (2g-2+n)s \pmod{r}.
\end{equation*}
This condition guarantees the existence of a line bundle whose $r$-th tensor power is isomorphic to $\omega_{\log}^{\otimes s}(-\sum_i a_i p_i)$. Let $\pi \colon \overline{\mathcal{C}}_{g}^{r,s}(a_1,\dots,a_n) \to \overline{\mathcal{M}}_{g}^{r,s}(a_1,\dots,a_n)$ be the universal curve, and $\mathcal{L} \to \overline{\mathcal C}_{g}^{r,s}(a_1,\dots,a_n)$ the universal $r$-th root. 
There is moreover a natural forgetful morphism
\begin{equation*}
	\epsilon \colon
	\overline{\mathcal{M}}^{r,s}_{g}(a_1,\dots,a_n)
	\longrightarrow
	\overline{\mathcal{M}}_{g,n}
\end{equation*}
which forgets the choice of the line bundle. 
It can be turned into an unramified covering in the orbifold sense of degree $2g - 1$ by slightly modifying the structure of $\overline{\mathcal{M}}_{g,n}$, introducing an extra $\ZZ_r$ stabilizer for each node of each stable curve (see~\cite{JPPZ}).

%
%
%
We can then consider the family of Chern classes 
\[
\tilde{\Omega}_{g,n}^{[x]}(r,s;a_1,\dots,a_n)
	\coloneqq 
	 \sum_{i\geq0}x^{i}c_i\left(-R^*\pi_*\mathcal{L} \right),
\]
where $x$ is a formal parameter, and its pushed forward to the moduli spaces of stable curves 
\begin{equation*}
	\Omega_{g,n}^{[x]}(r,s;a_1,\dots,a_n)
	\coloneqq 
	\epsilon_{\ast}  
	\left(\tilde{\Omega}_{g,n}^{[x]}(r,s;a_1,\dots,a_n)\right)
	\in
	R^*(\overline{\mathcal{M}}_{g,n}).
\end{equation*}
We will omit the variable $x$ when $x = 1$. Notice that we recover Mumford's formula for the Hodge class when $r = s = 1$, $x=1$, and $a = (1,\dots,1)$.  The indices $a_1,\dots,a_n$ are often referred to as ``primary fields''.

\begin{definition} We define the following classes for $a = a_1 + \dots + a_n$.
\begin{align}
\text{\hhh}_{g,n}(a_1, \dots, a_n) &\coloneqq a^{1-g} \Omega^{[a]}_{g,n}(a, 0; - a_1, \dots, - a_n) 
\end{align}
\end{definition}

\subsection{$\Omega$-classes properties} \label{sec:Omega-basic} This section recalls several properties of the $\Omega$-classes collected in~\cite{BLRS,Euler,lpsz}. Fix $g,n \geq 0$ integers such that $2g - 2 + n > 0$. Let $r$ and $s$ be integers with $r$ positive. Let $a_1, \ldots, a_n$ be integers satisfying the modular constraint 
$$
a_1+a_2+\cdots+a_n \equiv (2g-2+n)s \pmod{r}.
$$ 
Let $x$ be a formal variable. Then $\Omega$-classes satisfy the following properties:

\subsubsection{Primary fields properties} \label{properties:shift-ai}
We have:
\begin{enumerate}
\item $\Ch^{[x]}_{g,n}(r, s; a_1, \dots, a_i + r, \dots, a_n) = \Ch^{[x]}_{g,n}(r,s;a_1, \dots, a_n) \cdot\left( 1 + x\frac{a_i}{r}\psi_i\right)$
\item $\Ch^{[x]}_{g,n}(r, s+r; a_1, \dots, a_n) = \Ch^{[x]}_{g,n}(r, s; a_1, \dots, a_n) \cdot \exp\left( \sum_{m=1} \frac{(-x)^m}{m}\left( \frac{s}{r} \right)^m \kappa_m \right)$
\item $\Ch^{[x]}_{g,n}(r,0;a_1, \dots, a_n) = \Ch^{[x]}_{g,n}(r,r;a_1, \dots, a_n)$
\item $\Ch^{[x]}_{g,n}(r,s; a_1, \dots, 0, \dots, a_n) = \Ch^{[x]}_{g,n}(r,s;a_1, \dots, r, \dots, a_n)$
\item Let $n\geq 1$. Then $\lambda_{g} \Omega^{[x]}_{g,n}\big(r,0; 0,\ldots,0 \big) = r^{2g-1}\lambda_{g} \left( 1 -x\lambda_1 + \dots + (-x)^g \lambda_g \right)$
\end{enumerate}

\subsubsection{CohFT-like properties}\label{properties:pull-back-range} Assume in addition that $0\leq a_1,\dots,a_n\leq r$. Then 
\begin{enumerate}
\item[(6)] $\Ch^{[x]}_{g,n}(r,s;a_1, \dots, a_n, s) = \pi^{\ast}\Ch^{[x]}_{g,n}(r,s;a_1, \dots, a_n)$
\end{enumerate}
where $\pi\colon \oM_{g,n+1}\to \oM_{g,n}$ forgets the last marked point. 

\subsubsection{Polynomiality properties}\label{properties:poly} 
\begin{enumerate}
\item[(7)] Let $r$ to be a sufficiently large natural number. Then 
$$
\left(\Omega_{g,n}^{[x]}(r,0; a_1, \dots, a_n)\right)_k \Big{|}_{\oM_{g,n}^{ct}} \in R^{k}(\oM^{ct}_{g,n})
$$ 
is a non-homogeneous polynomial in the $a_i$ of polynomial degree equal to $k$.
\end{enumerate} 

\subsubsection{Riemann-Roch vanishing}\label{properties:RR} 
\begin{enumerate}

\item[(8)] Let $a$ be a positive integer, let $b_1, \dots, b_n$ be integers such that their sum is positive and divisible by $a$, let $B$ be the quotient. Then we have the following bound on the cohomological degree:
$$
\deg \left(\Omega_{g,n}(a, 0; b_1, \dots, b_n)\right) \leq g + B - 1.
$$

\end{enumerate}

\section{The \texorpdfstring{$A=\bOmT$}{A=Omega} relation}\label{sec:AOmega}

The goal of this section is to recall a conjecture in~\cite{BLRS}. We have to introduce a notation for the trees and we follow the op. cit. verbatim.

\subsection{Basic notation for trees}
Let $\SRT_{g,n,m}$ be the set of stable rooted trees of total genus $g$, with $n$ regular legs $\sigma_1,\dots,\sigma_n$ and $m$ extra legs $\sigma_{n+1},\dots,\sigma_{n+m}$, which we refer to as ``frozen'' legs and must always be attached to the root vertex. For a $T\in \SRT_{g,n,m}$ we use the following notation:
\begin{itemize}
	\item $H(T)$ is the set of half-edges of $T$.
	\item $L(T),L_r(T),L_f(T)\subset H(T)$ are the sets of all, regular, and frozen legs of $T$, respectively. $L(T) = L_r(T)\sqcup L_f(T)$.
	\item $H_e(T)\coloneqq H(T)\setminus L(T)$.
	\item $\iota\colon H_e(T)\to H_e(T)$ is the involution that interchanges the half-edges that form an edge.
	\item $E(T)$ is the set of edges of $T$, $E\cong H_e(T)/\iota$.
	\item $H_+(T)\subset H(T)$ is the set of the so-called ``positive'' half-edges that consists of all regular legs of $T$ and of half-edges in $H(T)\setminus L(T)$ directed away from the root at the vertices where they are attached,
	$H_+(T)\cong E(T)\cup L_{r}(T)$; 
	\item $H_-(T)\subset H(T)$ is the set of the so-called ``negative'' half-edges that consists of all frozen legs of $T$ and of half-edges in $H(T)\setminus L(T)$ directed towards the root at the vertices where they are attached, $H_-(T)\cong E(T)\cup L_{f}(T)$;
	\item $V(T),V_{nr}(T)$ are the sets of vertices and non-root vertices of $T$. 
	\item $v_r\in V(T)$ is the root vertex of $T$; $V(T)=\{v_r(T)\}\sqcup V_{nr}(T)$.
	\item For a $v\in V(T)$, $H(v),H_+(v),H_-(v)$ are all, positive, and negative half-edges attached to $v$, respectively. Obviously, $|H_-(v_r)|=m$ and for any $v\in V_{nr}(T)$ we have $|H_-(v)|=1$.
	\item For a $v\in V(T)$ let $g(v)\in \ZZ_{\geq 0}$ be the genus assigned to $v$. The stability condition means that $2g(v)-2+|H(v)|>0$.
	The genus condition reads $\sum_{v\in V(T)} g(v) = g.$
	\item We say that a vertex or a (half-)edge $x$ is a descendant of a vertex or a (half-)edge $y$ if $y$ is on the unique path connecting $x$ to $v_r$. 
	\item For an $h\in H_+(T)$ let $DL(h)$ be the set of all legs that are descendants to $h$, including $h$ itself. Note that $DL(h)\subseteq L_r(T)$ for any $h\in H_+(T)$ and $DL(l)=\{l\}$ for $l\in L_r(T)$. 
	\item For an $e\in E(T)$ let $DL(e)$ be the set of all legs that are descendants to $e$. Note that $DL(e)\subseteq L_r(T)$ for any $e\in E(T)$. 
	\item For an $v\in V(T)$ let $DL(v)$ be the set of all regular legs that are descendants to $v$. In particular, $DL(v_r) = L_r(T)$. 
	\item For a $v\in V(T)$ let $DV(v)\subset V(T)$ be the subset of all vertices that are descendants of $v$, including $v$ itself. For instance, $DV(v_r) = V(T)$. 
\end{itemize}

\subsubsection{Ring of coefficients}
Consider the polynomial ring $Q\coloneqq \QQ[a_1,\dots,a_n]$ and define $a\colon H_+(T) \to Q$, $a\colon E(T)\to Q$, and $a\colon V(T)\to Q$ (abusing notation we use the same symbol $a$ for all of these maps) by 
\begin{align}
	a(\sigma_i)& \coloneqq a_i, &i=1,\dots,n; 
	& & a(h)& \coloneqq \textstyle\sum_{l\in DL(h)} a(l), & h\in H_+(T); \\
	a(e)& \coloneqq \textstyle\sum_{l\in DL(e)} a(l), & e\in E(T); 
	& & a(v)& \coloneqq \textstyle\sum_{l\in DL(v)} a(l), & v\in V(T).
\end{align} 


\subsubsection{Levels}
We enhance the structure of a stable rooted tree to the so-called leveled stable rooted tree (of genus $g$, with $n$ regular and $m$ frozen legs). Let $T\in\SRT_{g,n,m}$. A function $\ell\colon V(T)\to\ZZ_{\geq 0}$ is called a level function if the following conditions are satisfied:
\begin{itemize}
	\item The value of $\ell$ on the root vertex is zero ($\ell(v_r) = 0$).
	\item If $v'\in DV(v)$ and $v'\not=v$, then $\ell(v')>\ell(v)$. 
	\item There are no empty levels, that is, for any $0\leq i \leq \max \ell(V(T))$ the set $\ell^{-1}(i)$ is non-empty. 
\end{itemize}
 A leveled stable rooted tree is a pair $(T,\ell)$. The height of a leveled stable rooted tree $(T,\ell)$ is $\ell(T)\coloneqq \max \ell(V(T))$. Let $\mathcal{L}(T)$ denote the set of level functions on $T$. 

Once we fix a level function $\ell\in \mathcal{L}(T)$, we extend the definition of a genus function $g\colon V(T) \to \ZZ_{\geq 0}$ to the level genus function $g^{\mathrm{lvl}}\colon \{0,\dots,\ell(T)\}\to \ZZ_{\geq 0}$ defined as 
\begin{align}
	g^{\mathrm{lvl}} (l) \coloneqq \sum_{v\in V(T), \ell(v) \leq l} g(v). 
\end{align}
Of course, 	$g^{\mathrm{lvl}} (\ell(T))=g$. 

\subsubsection{Degree}
We enhance the structure of a stable rooted tree $T\in SRT_{g,n,m}$ with a so-called degree function $d\colon V(T) \to \ZZ_{\geq 0}$. 
Let $\mathcal{D}(T)$ denote the set of degree labels on $T$. 

In the presence of a level function $\ell\in \mathcal{L}(T)$, we associate with a degree function $d$  a level degree function $d^{\mathrm{lvl}}\colon \{0,\dots,\ell(T)\}\to \ZZ_{\geq 0}$, where 
\begin{align}
	d^{\mathrm{lvl}} (l) \coloneqq \sum_{v\in V(T), \ell(v) \leq l} d(v) + |\{v\in V(T), \ell(v) \leq l\}|-1. 
\end{align}
Define also 
$
d(T)\coloneqq d^{\mathrm{lvl}} (\ell((T)).
$

\subsection{\texorpdfstring{$A-\bOmT$}{A-Omega} vanishing relations} \label{sec:vanishing}

\begin{definition}
For any $g,n,m\geq 0$ such that $2g-2+n+m>0$ and for any $a_1,\dots,a_{n}$, $a\coloneqq \sum_{i=1}^{n} a_i$, let 
\begin{align}
	\bOm_{g,n}^{m}(a_1,\dots,a_{n}) &\coloneqq \left(\pi^{*}\right)^{m} \left(\lambda_g \hhh_{g,n}(a_1, \dots, a_n)\right)
	\\ \notag
	&= (a)^{1-g}\lambda_{g} \Omega^{[a]}_{g,n+m}\Big(a,0; -a_1,\dots,-a_{n}\underbrace{0,\ldots,0}_{m}\Big)
\end{align}
(the map $\pi$ here forgets the last marked point at each application, so the power $m$ forgets $m$ distinct points).
\end{definition}

Let $T\in\SRT_{g,n,m}$. Assign to each $v\in V(T)$ the moduli space of curves $\oM_{g(v),|H(v)|}$, where the first $|H_+(v)|$ marked points correspond to the positive half-edges attached to $v$ and ordered in an arbitrary but fixed way and the the last $|H_-(v)|$ marked points correspond to the negative half-edges attached to $v$, also ordered in some arbitrary but fixed way. Consider the class
\begin{align}
& \bOm(v)  \coloneqq \bOm^{|H_-(v)|}_{g(v),|H_+(v)|} (a(h_1),\dots,a(h_{|H_+(v)|}))
\in R^*(\oM_{g(v),|H(v)|})\otimes_{\QQ}Q.
\end{align}


%

\begin{definition}
	For each $(g,n,m)$ such that $2g-2+n+m>0$ define the class 
	$$\bOmT^m_{g,n}\in R^*(\oM_{g,n+m})\otimes_{\QQ}Q$$
	as
	\begin{equation}
		\bOmT^m_{g,n} \coloneqq \sum_{\substack{T\in\SRT(g,n,m), d\in \mathcal{D}(T), \ell\in \mathcal{L}(t)\\ \forall i < \ell(T)\colon d^{\mathrm{lvl}}(i)\leq 2g^{\mathrm{lvl}}(i)-2+m}} (-1)^{\ell(T)} \biggl(\prod_{e\in E(T)} a(e)\biggr) (b_T)_* \bigotimes_{v \in V(T)} (\bOm(v))_{d(v)}
	\end{equation}
	Here $(b_T)_*$ is the boundary pushforward map that acts from $\bigotimes_{v \in V(T)} R^*(\oM_{g(v),|H(v)|})\otimes_{\QQ} Q$ to $R^*(\oM_{g,n+m})\otimes_{\QQ}Q$. 
	The component of the class $\bOmT^m_{g,n}$ in $R^p$ is homogeneous polynomials of degree $p$ in $a_1,\dots,a_n$, $p=0,\dots,3g-3+n+m$. 
\end{definition}

Consider the moduli space 
\begin{align}
\overline{\mathcal{M}}_{g}^{\sim}(\mathbb{P}^1,a_1,\dots,a_n,-\sum_{i=1}^n a_i)	
\end{align}
of rubber stable maps to $(\mathbb{P}^1,0,\infty)$. Let $s\colon \oM_{g}^{\sim}(\mathbb{P}^1,a_1,\dots,a_n,-\sum_{i=1}^n a_i)\to \oM_{g,n+1}$ be the projection to the source curve, and $\lambda_g$ the lift of the lambda class with respect to this projection. Let $t\colon \oM_{g}^{\sim}(\mathbb{P}^1,a_1,\dots,a_n,-\sum_{i=1}^n a_i)\to LM_{2g-1+n}$ be the projection to the target curve, where $LM_{2g-1+n}$ denotes the Losev-Manin space with $2g-1+n$ marked points. Let $\tilde \psi_0$ be the pull-back by $t$ of the $\psi$-class at the point $0$ in Losev-Manin space. Define
\begin{align} \label{eq:NewA1}
	A^1_{g,n} \coloneqq s_* \left(\frac{\lambda_g}{1-\tilde\psi_0} \left[\oM_{g}\left(\mathbb{P}^1,a_1,\dots,a_n,-\sum\nolimits_{i=1}^n a_i\right)\right]^{\mathrm{vir}}\right)\in R^*(\oM_{g,n+1})\otimes_{\QQ}Q.
\end{align}
Note that $(A^1_{g,n})_{<2g}=0$ and $(A^1_{g,n})_{2g}=\lambda_g\DR_{g}(a_1,\dots,a_n,-\sum_{i=1}^n a_i)$. Note also that for any $d\geq 2g$ the class $(A^1_{g,n})_{d}$ is a homogeneous polynomial of degree $d$ in $a_1,\dots,a_n$.  

Also, let $A^0_{g,n}\coloneqq\frac{1}{\sum_{i=1}^n a_i} \pi_*A^1_{g,n}$. Note that $(A^0_{g,n})_{<2g-1}=0$ for any $d\geq 2g-1$ the class $(A^0_{g,n})_{d}$ is a homogeneous polynomial of degree $d$ in $a_1,\dots,a_n$.  

\begin{remark}
The definitions of the $A$-classes differ from the original ones in \cite{BGR19,BS22}, they are however equivalent due to \cite[Lemma 2.2]{BLS}.
\end{remark}

\begin{theorem}[Main Theorem]\label{thm:Omega-A} For $g,n,m\geq 0$, $2g-2+n+m>0$, we have 
\begin{equation}
\deg \left(\bOmT^m_{g,n} - A_{g,n}^{m}(\delta_{m,0} + \delta_{m,1}) \right) \leq 2g-2+m. 
\end{equation}
\end{theorem}

The degree here can be either cohomological degree or the degree of a polynomial in $a_1,\dots,a_n$. To be specific, let's use the degree in $a_i$'s. Note that the statement of the theorem for $m=0$ follows from the statement for $m=1$ by push-forward. 

\section{Virtual localisation}\label{sec:virloc}

\subsection{A relation from the virtual localization}

For any $g\geq 0, n\geq 1$ such that $2g-2+n \geq 0$ and for any $a_1,\dots, a_{n}$, $a \coloneqq\sum_{i=1}^{n} a_i$, let 
\begin{align}
	\bD_{g,n+1}(a_1,\dots,a_{n})& \coloneqq -\frac{\lambda_{g}\DR_{g}\big(a_1,\dots,a_{n},-a\big)}{(1+a\psi_{n+1})} 
\end{align}
If $T\in \SRT_{g,n,m}$, $v\in V_{nr}(T)$, then
\begin{align}
	\bD(v)& \coloneqq \bD_{g(v),|H(v)|}(a(h_1),\dots,a(h_{|H_+(v)|}))
	\in R^*(\oM_{g(v),|H(v)|})\otimes_{\QQ}Q. 
\end{align}

For $m\geq1$, we consider the following class: 
\begin{align} \label{eq:Definition-Upsilon}
	\bUp^m_{g,n}(a_1,\dots,a_n) & \coloneqq \delta_{m,1} \bD_{g,n+1}(a_1,\dots,a_{n}) + \bOm_{g,n}^{m}(a_1,\dots,a_{n})  
	\\ \notag & \quad 
	+ \sum_{\substack{T\in\SRT(g,n,m)\\ \ell\in \mathcal{L}(t),\, \ell(T)=1}} \biggl(\prod_{e\in E(T)} a(e)\biggr) (b_T)_* \bigg( \bOm(v_r)\otimes \bigotimes_{v \in V_{nr}(T)} \bD(v) \bigg).
\end{align}

\begin{theorem} 
\label{thm:localization} 
We have $\deg(\bUp^m_{g,n})_d\leq 2g-2+m$. 
\end{theorem}

The rest of this section is devoted to the proof of this theorem. 

\subsection{Stable relative maps to orbifold target}

Let $g$ and $m$ be two nonnegative integers. Let $a_{1},\dots,a_{n}$ be positive integers and denote by $a$ their sum. Let $\mathbb{P}^{1}\left[a\right]$ be the projective line with a single orbifold point $B\mathbb{Z}_{a}$ at $\infty\in\mathbb{P}^{1}$. Let
\[
\overline{\mathcal{M}}_{g,m}\left(\mathbb{P}^{1}\left[a\right],a_{1},\dots,a_{n}\right)
\]
be the moduli space stable maps to the orbifold/relative pair $\left(\mathbb{P}^{1}\left[a\right],0\right)$. This moduli space parametrizes maps
\[
f:C\rightarrow P
\]
where the source $C$ is a twisted curve of genus $g$, the target $P$ is a glueing of a chain of $l\geq0$ copies of $\mathbb{P}^{1}$ to $0\in\mathbb{P}^{1}\left[a\right]$, the map $f$ is a twisted stable map and the ramification profile over $0$ is given by the partition $\left(a_{1},\dots,a_{n}\right)$, in particular $f$ is of degree $a$. The $n$ ramifications points are marked from $1$ to $n$ and there are $m$ additional marked points from $n+1$ to $n+m$. In this setting, all the orbifold points of the twisted curve $C$ appear at the nodes in the preimage of the orbifold point at $\infty$ by $f$. The restriction of $f$ to the chain of $\mathbb{P}^{1}$ belongs to a moduli space of rubber stable maps and is called the rubber part of $f$. We refer to \cite{AGV08, JPT11} for a detailed definition of the moduli space of twisted stable maps to $\left(\mathbb{P}^{1}\left[a\right],0\right)$. This moduli space has a virtual fundamental class of virtual dimension 
\[
\textrm{vdim} \left[\overline{\mathcal{M}}_{g,m}\left(\mathbb{P}^{1}\left[a\right],a_{1},\dots,a_{n}\right)\right]^{{\rm vir}}=2g-1+n+m.
\]

\subsection{$\mathbb{C}^{*}$-action and fixed loci}

We consider the $\mathbb{C}^{*}$-action on $\mathbb{P}^{1}$ given by
\[
t\cdot\left[z_{0}:z_{1}\right]=\left[z_{0}:tz_{1}\right]
\]
which lifts to $\mathbb{P}^{1}\left[a\right]$ and to $\overline{\mathcal{M}}_{g,m}\left(\mathbb{P}^{1}\left[a\right],a_{1},\dots,a_{n}\right)$.

Each $\mathbb{C}^{*}$-fixed locus of $\overline{\mathcal{M}}_{g,m}\left(\mathbb{P}^{1}\left[a\right],a_{1},\dots,a_{n}\right)$ is labelled by a bipartite graph $\Phi$ such that each vertex $v\in V\left(\Phi\right)$ is decorated by a genus $g\left(v\right)$ and is labelled by $0\in\mathbb{P}^{1}\left[a\right]$ or $\infty\in\mathbb{P}^{1}\left[a\right]$, we denote by $V_{0}\left(\Phi\right)$ the set of vertices over $0$ and $V_{\infty}\left(\Phi\right)$ the vertices over $\infty$. In addition, each edge $e\in E\left(\Phi\right)$ is decorated with a nonzero degree $d_{e}$. Finally, there are $n+m$ legs: the $n$ first legs are called regular legs, and the $m$ remaining legs are called frozen legs. 

We associate to a fixed map $f:C\rightarrow P$ a bipartite graph $\Phi$ such that: 
\begin{itemize}
\item each degree $d_{e}$ covering of $\mathbb{P}^{1}\left[a\right]$ by a genus $0$ curve (with possibly an orbifold point $B\mathbb{Z}_a$ over $\infty$) gives an edge decorated by $d_{e}$, 
\item each connected component in the preimage of $\infty$ gives a vertex $v$ labelled by $\infty$, 
 it can be either:
	\begin{itemize}
	\item a twisted curve, in that case $g(v)$ is the genus of that twisted curve,
	\item an orbifold point appearing at the node of two genus $0$ covers of $\mathbb{P}^{1}\left[a\right]$, in that case $g(v)=0$,
	\item a smooth point on a genus $0$ cover of $\mathbb{P}^{1}\left[a\right]$, in that case $g(v)=0$,
	\end{itemize}
\item if the target expands, that is if there is at least one $\mathbb{P}^1$ glued to $\mathbb{P}^1[a]$ in the target of $f$, each disconnected component of the source of the rubber map gives a vertex $v$ over $0$, and $g(v)$ denotes its genus,
\item if the target does not expand, each of the $n$ points in the ramification profile over $0$ gives a vertex over $0$ of genus $0$,
\item the $n$ ramification points give the $n$ regular legs, and the $m$ marked points give the $m$ frozen legs.
\end{itemize}
A vertex $v$ is stable if $$2g(v)-2+\left|E_{v}\right|+m(v)+n(v)>0,$$ where $|E_v|$ is the cardinal of the set of edges attached to $v$,  $n(v)$ denotes the number of regular legs attached to $v$ and $m(v)$ denotes the number of frozen legs attached to $v$.

Conversely, the maps of the fixed locus $\Phi$ are obtained by glueing the maps associated to the vertices over $0$ and $\infty$ with the degree $d_e$ coverings. More precisely, we describe the fixed locus $\Phi$ in the following way. We denote by $V^{st}_{\infty}\left(\Phi\right)$ the stable vertices over $\infty$. Let 
\[
\overline{\mathcal{M}}_{\Phi}=\begin{cases}
\left(\prod_{v\in V_{0}\left(\Phi\right)}\overline{\mathcal{M}}_{v}\right)^{\sim}\times\prod_{v\in V^{st}_{\infty}\left(\Phi\right)}\overline{\mathcal{M}}_{v}^{a} & {\rm if\,the\,target\,expands}\\
\prod_{v\in V^{st}_{\infty}\left(\Phi\right)}\overline{\mathcal{M}}_{v}^{\frac{1}{a}} & {\rm if\,the\,target\,does\,not\,expand}
\end{cases}
\]
where
\begin{itemize}
\item $\overline{\mathcal{M}}_{v}^{a}$, for $v\in V^{st}_{\infty}\left(\Phi\right)$, stands for the moduli space of $a$-spin structures
\[
\overline{\mathcal{M}}_{g\left(v\right)}^{a,0}\Big(a-d_{e_{1}},\dots,a-d_{e_{\left|E_{v}\right|}},\underset{m\left(v\right)}{\underbrace{0,\dots,0}}\Big)
\]
 where $e_{1},\dots,e_{\left|E_{v}\right|}$ are the indices of the edges attached to $v$,
\item $\overline{\mathcal{M}}_{v}$, for $v\in V_{0}\left(\Phi\right)$, stands for the moduli space
\[
\overline{\mathcal{M}}_{g\left(v\right),m\left(v\right)}\left(a_{i_{1}^{v}},\dots,a_{i_{\left|L_{r}^v\right|}^{v}},-d_{e_{1}},\dots,-d_{e_{\left|E_{v}\right|}}\right),
\]
of relative (not rubber) maps to the pair $\left(\mathbb{P}^{1},0\cup\infty\right)$, where $L_{r}^v=\left\{ i_{1}^{v},\dots,i_{\left|L_{r}^{v}\right|}^{v}\right\} \subseteq\left\{ 1,\dots,n\right\} $ is the set of indices of the regular legs attached to $v$,
\item the tilde indicates that two products of maps are identified if they differ by an element of $\mathbb{C}^{*}$ in each $\mathbb{P}^{1}$ of the target.
\end{itemize}
We recall that the moduli of spin structures $\overline{\mathcal{M}}^{a,s=0}_{g}(a_1,\dots,a_n)$ is also the moduli space of twisted stable maps to the orbifold point $B\mathbb{Z}_a $ such that the source curve has, in addition to the orbifold points at the nodes, $n$ orbifold points with orbifold structure determined by $a_1,\dots,a_n$. Thus, there is a map
\[
\iota:\overline{\mathcal{M}}_{\Phi}\rightarrow\overline{\mathcal{M}}_{g,m}\left(\mathbb{P}^{1}\left[a\right],a_{1},\dots,a_{n}\right)
\]
glueing the different contributions of $\Phi$,  the automorphism group of $\iota_{*}\left(\overline{\mathcal{M}}_{\Phi}\right)$ is of order
\[
\left|{\rm Aut}\left(\Phi\right)\right|\prod_{e\in E\left(\Phi\right)}d_{e}, 
\]
and the class of the fixed locus associated to $\Phi$ is $\frac{\iota_{*}\left(\left[\overline{\mathcal{M}}_{\Phi}\right]^{{\rm vir}}\right)}{\left|{\rm Aut}\left(\Phi\right)\right|\prod_{e\in E\left(\Phi\right)}d_{e}}$. 
\begin{lemma}
	The graph $\Phi$ of a fixed locus has a unique (stable or unstable) vertex over $\infty\in\mathbb{P}^{1}\left[a\right]$.
\end{lemma}
\begin{definition}
	This unique vertex is called the \emph{root} of the graph.
\end{definition}

\begin{proof}
Let $v\in V_{\infty}\left(\Phi\right)$, there are three possibilities.

\begin{enumerate}
	\item  If $v$ is an unstable vertex with one incoming edge and possibly one frozen leg, then $v$ corresponds to a smooth point on a genus $0$ component of the twisted curve $C$, and this genus $0$ component is a cover of a degree $d_{e}$ of $\mathbb{P}^{1}\left[a\right]$. Now, an orbifold cover $\mathbb{P}^{1}\rightarrow\mathbb{P}^{1}\left[a\right]$ of degree $d_{e}$ only exists if $d_{e}$ is a nonzero multiple of $a$, and since this degree $d_{e}$ is lower or equal to the total degree $a$ we get $d_{e}=a$. Thus, there is no other edge and no other vertex above $\infty$. 
	\item If $v$ is an unstable vertex with two incoming edges, one of degree $d_{1}$ and the other of degree $d_{2}$, then $v$ corresponds to an orbifold node with stabilizer $\mathbb{Z}_a$ and the local picture of a twisted curve at a node imposes that $d_{2}=a-d_{1}$. Thus, there are no other edge and in particular no other vertex above $\infty$.
	\item If $v$ is a stable vertex with $\left|E_{v}\right|$ incoming edges of degrees $d_{1},\dots,d_{\left|E_{v}\right|}$ with possibly some frozen legs, then the modular condition of the moduli space of spin structures $\overline{\mathcal{M}}_{v}^{a}$ is
\[
\left(a-d_{1}\right)+\cdots+\left(a-d_{\left|E_{v}\right|}\right)\equiv0\quad\textrm{mod}\;a.
\]
Since the total sum of the degrees equals $a$, this condition is satisfied only if each edge is adjacent to $v$, in particular there is no other vertex above $\infty$. 
\end{enumerate}
\end{proof}

\subsection{The virtual localization formula}

Let $u=c_{1}\left(\mathcal{O}_{\mathbb{P}^{\infty}}\left(-1\right)\right)$ be a generator of the equivariant cohomology of the point, that is,
\[
H_{\mathbb{C}^{*}}^{*}\left(pt\right)=\mathbb{Z}\left[u\right].
\]
The virtual localization formula \cite{GP99, GV05} gives the following equality in 
$$
H^{*}\left(\overline{\mathcal{M}}_{g,m}\left(\mathbb{P}^{1}\left[a\right],a_{1},\dots,a_{n}\right)\right)\otimes\mathbb{Q}[u,u^{-1}]
$$
by expressing the virtual fundamental class as the sum of the fixed weighted loci:
\begin{equation}
\left[\overline{\mathcal{M}}_{g,m}\left(\mathbb{P}^{1}\left[a\right],a_{1},\dots,a_{n}\right)\right]^{{\rm vir}}=\sum_{\Phi}\frac{1}{\left|{\rm Aut}\Phi\right|}\frac{1}{\prod d_{e}}\iota_{*}\left(\frac{\left[\overline{\mathcal{M}}_{\Phi}\right]^{{\rm vir}}}{e_{\mathbb{C}^{*}}\left(\mathcal{N}_{\Phi}^{{\rm vir}}\right)}\right),\label{eq: loca formula}
\end{equation}
where $\mathcal{N}_{\Phi}^{{\rm vir}}$ is the virtual normal bundle of the fixed locus $\Phi$ in $\overline{\mathcal{M}}_{g,m}\left(\mathbb{P}^{1}\left[a\right],a_{1},\dots,a_{n}\right)$ and we denote by $e_{\mathbb{C}^{*}}\left(\mathcal{N}_{\Phi}^{{\rm vir}}\right)$ its equivariant Euler class. The term $\iota_{*}\left(\left[\overline{\mathcal{M}}_{\Phi}\right]^{{\rm vir}} / e_{\mathbb{C}^{*}}\left(\mathcal{N}_{\Phi}^{{\rm vir}}\right)\right)$ can be computed (for more details see \cite{JPT11,JPPZ}) in the following way.
\begin{itemize}
\item If the target does not expand it contributes a factor of $1$, otherwise it contributes the factor
\[
-\frac{\prod_{e\in E\left(\Phi\right)}d_{e}}{u+\tilde{\psi}_{\infty}}\left[\left(\prod_{v\in V_{0}\left(\Phi\right)}\overline{\mathcal{M}}_{v}\right)^{\sim}\right]^{{\rm vir}}
\]
where the class $\tilde{\psi}_{\infty}$ is pull-back by the forgetful map keeping the target curve of the $\psi$-class at the point $\infty$ in the Losev-Manin space.
\item The vertex over $\infty$ contributes in the following way:
\begin{itemize}
\item if the vertex is stable, as
\[
\prod_{e\in E\left(\Phi\right)}d_{e}\times\left(\frac{a}{u}\right)^{2-g}\tilde{\Omega}_{g\left(v\right), |E(v)| + m(v)}^{\left[\frac{a}{u}\right]}\Big(a,0;-d_{e_{1}},\dots,-d_{e_{\left|E_{v}\right|}},\underset{m\left(v\right)}{\underbrace{0,\dots,0}}\Big),
\]
where the class $\tilde{\Omega}$, defined in Section~\ref{sec:Omega:classes}, is such that its push forward by the forgetful map from the moduli space of spin structures to the moduli space of curves gives the $\Omega$-class,
\item if the vertex is unstable and has one incoming edge, a factor of $1$,
\item if the vertex is unstable and has one incoming edge and one frozen leg, a factor of $\frac{a}{u}$,
\item if the vertex is unstable and has two incoming edges, a factor of $\frac{d_{1}d_{2}}{u}$.
\end{itemize}
\end{itemize}
\begin{remark}
	Note that we used \cite[Theorem 4.1 (ii)]{Euler}, valid for the $\Omega$-class but also for the $\tilde{\Omega}$-class, to write the contribution of the stable vertex over $\infty$ in a compact way.
\end{remark}

\subsection{Proof of Theorem~\ref{thm:localization}}

Fix $m\geq1$. Let ${\rm ev}_{i}:\overline{\mathcal{M}}_{g,m}\left(\mathbb{P}^{1}\left[a\right],a_{1},\dots,a_{n}\right)\rightarrow\mathbb{P}^{1}\left[a\right]$ be the evaluation map corresponding to the $i$-th marked point and let $\left[\infty\right]\in H_{\mathbb{C}^{*}}^{*}\left(\mathbb{P}^{1}\left[a\right]\right)$ be the $\mathbb{C}^{*}$-equivariant cohomology class dual to the point $\infty\in\mathbb{P}^{1}\left[a\right]$. We prove Theorem~\ref{thm:localization} with the following sequence of steps.
\begin{itemize}
    \item We first intersect the virtual localization formula (\ref{eq: loca formula}) with the class $\prod_{i=n+1}^{n+m}{\rm ev}_{i}^{*}\left(\left[\infty\right]\right).$ Notice that this intersection selects graphs such that all the frozen legs are attached to the root.
    \item Then, we extract the negative power of $u$ on both sides of this equation. Since the LHS $\left[\overline{\mathcal{M}}_{g,m}\left(\mathbb{P}^{1}\left[a\right],a_{1},\dots,a_{n}\right)\right]^{{\rm vir}}\prod_{i=n+1}^{n+m}{\rm ev}_{i}^{*}\left(\left[\infty\right]\right)$ is a class in the equivariant cohomology of $\overline{\mathcal{M}}_{g,m}\left(\mathbb{P}^{1}\left[a\right],a_{1},\dots,a_{n}\right)$, it has no negative power of $u$. Thus extracting the negative powers of $u$ gives a list of relations.
    \item Then, we push each relation to $\overline{\mathcal{M}}_{g,n+m}$ by the morphism $s:\overline{\mathcal{M}}_{g,m}\left(\mathbb{P}^{1}\left[a\right],a_{1},\dots,a_{n}\right)\rightarrow\overline{\mathcal{M}}_{g,n+m}$ forgetting the target and preserving the source curve.
    \item Finally, we intersect each relation with $\lambda_{g}$, since this class vanishes on the complement of the moduli space of compact type,  the only graphs surviving are trees.
\end{itemize}
Thus, for each negative power of $u$ we obtain a relation in $\overline{\mathcal{M}}_{g,n+m}$ such that the graphs involved are trees with a root (stable or unstable) over $\infty$, such that all frozen legs are attached to the root, and there are $\left|V_{0}\left(\Phi\right)\right|$ other vertices connected to the root by a unique edge, in addition the $n$ regular legs are distributed among the vertices over $0$. We classify these trees in three types when $m=1$ and two types when $m\geq2$:
\begin{enumerate}
\item (DR type, only for $m=1$) The DR tree has one unstable vertex over $\infty$ of genus~$0$ with one frozen leg, one stable vertex of genus $g$ over $0$ where the $n$ regular legs are attached, and a unique edge connecting these vertices.  Note that this tree disappear when $m\geq2$ since the vertex over $\infty$ becomes stable.
\item ($\Omega$ type) The $\Omega$ tree has a unique stable vertex of genus $g$ over $\infty$ with $m$ frozen legs and $n$ edges attached to it, each edge is connected to an unstable vertex with one regular leg.
\item (Mixed type) A mixed type tree has a unique stable vertex over $\infty$ and at least one stable vertex over $0$. 
\end{enumerate}
These $3$ types of trees will correspond to the $3$ types of terms in formula~(\ref{eq:Definition-Upsilon}). 

\subsubsection{Dealing with $\tilde{\psi}_{\infty}$. }

The following lemma explains how to take care of the class $\tilde{\psi}_{\infty}$ when we push forward by $s$. Notice that, since we are only interested in rooted trees such that all the frozen legs are attached to the root, we directly simplified the shape of the disconnected components of the moduli space of rubber maps contributing to the vertices over $0$.
\begin{lemma}
\label{lem:Psi infinity}
Let $k\geq0$. The class
\[
s_{*}\left(\tilde{\psi}_{\infty}^{k}\left[\left(\prod_{v\in V_{0}\left(\Phi\right)}\overline{\mathcal{M}}_{g\left(v\right)}\left(a_{i_{1}^{v}},\dots,a_{i_{\left|L_{r}^{v}\right|}^{v}},-\sum_{j\in L_{r}^{v}}a_{j}\right),\right)^{\sim}\right]^{{\rm vir}}\right)
\]
vanishes if there is an unstable vertex in $V_{0}\left(\Phi\right)$, if not it equals
\[
\begin{cases}
0 & {\rm if}\:k<\left|V_{0}\left(\Phi\right)\right|-1,\\
\sum\left(\prod_{v\in V_{0}\left(\Phi\right)}\psi_{\left|L_{r}^{v}\right|+1}^{k_{v}}{\rm DR}_{g\left(v\right)}\left(a_{i_{1}^{v}},\dots,a_{i_{\left|L_{r}^{v}\right|}^{v}},-\sum_{j\in L_{r}^{v}}a_{j}\right)\right) & {\rm if}\:k\geq\left|V_{0}\left(\Phi\right)\right|-1,
\end{cases}
\]
where the sum is over the indices $k_v$ such that $\sum_{v \in V^0(\Phi)}  k_v=k-(\left|V_{0}\left(\Phi\right)\right|-1)$.
\end{lemma}

As a consequence, we get
\begin{align*}
 & s_{*}\left(-\frac{1}{u+\tilde{\psi}_{\infty}}\left[\left(\prod_{v\in V_{0}\left(\Phi\right)}\overline{\mathcal{M}}_{g\left(v\right)}\left(a_{i_{1}^{v}},\dots,a_{i_{\left|L_{r}^{v}\right|}^{v}},-\sum_{j\in L_{r}^{v}}a_{j}\right)\right)^{\sim}\right]^{{\rm vir}}\right)\\
 & =\frac{1}{u^{\left|V_{0}\left(\Phi\right)\right|}}\prod_{v\in V_{0}\left(\Phi\right)}-\frac{{\rm DR}_{g\left(v\right)}\left(a_{i_{1}^{v}},\dots,a_{i_{\left|L_{r}^{v}\right|}^{v}},-\sum_{j\in L_{r}^{v}}a_{j}\right)}{\left(1+\left(\sum_{j\in L_{r}^{v}}a_{j}\right)u^{-1}\psi_{\left|L_{r}^{v}\right|+1}\right)}.
\end{align*}

\begin{proof}
Let $t:\left(\prod_{v\in V_{0}\left(\Phi\right)}\overline{\mathcal{M}}_{g\left(v\right)}\left(a_{i_{1}^{v}},\dots,a_{i_{\left|L_{r}^{v}\right|}^{v}},-\sum_{j\in L_{r}^{v}}a_{j}\right)\right)^{\sim}\rightarrow LM_{r}$ be the forgetful map keeping only the target curve, where $r$ is the number of ramification points (counted with multiplicities and excluding the points over $0$ and $\infty$) and $LM_{r}$ denotes the Losev Manin space with $r$ points. We use the expression of the $\psi$-class at the point $\infty$ in $LM_{r}$ given in \cite[Eq.(2)]{BSSZ} to express its powers as
\begin{align}
\label{eq:psi_infty}
\psi_{\infty}^{k}=\sum_{I_{1}\sqcup\cdots\sqcup I_{k+1}=\left\llbracket r\right\rrbracket }\Delta_{I_{1},\dots,I_{k+1}}\frac{i_{k+1}}{i_{1}+\cdots+i_{k+1}}\frac{i_{k}}{i_{1}+\cdots+i_{k}}\cdots\frac{i_{1}}{i_{1}},
\end{align}
where we used the notation $i_{j}:=\left|I_{j}\right|$, and where $\Delta_{I_{1},\dots,I_{k+1}}$ is the boundary strata with $k+1$ components such that the component of $0$ contains the points of $I_{1}$, the next component contains the points of $I_{2}$ and so on. 

Let $v\in V_{0}\left(\Phi\right)$, we denote by $J_{v}\subset\left\llbracket r\right\rrbracket $ the set containing the image of the ramification points of the component $v$ by the rubber map. We obviously have $\amalg_{v\in V_{0}\left(\Phi\right)}J_{v}=\left\llbracket r\right\rrbracket $. It follows by \cite[Lemma 2.3]{BSSZ} that
\begin{align}
\label{eq: push forward divisor}
s_{*}\left(t^{*}\left(\Delta_{I_{1},\dots,I_{k+1}}\right)\left[\left(\prod_{v\in V_{0}\left(\Phi\right)}\overline{\mathcal{M}}_{g\left(v\right)}\left(a_{i_{1}^{v}},\dots,a_{i_{\left|L_{r}^{v}\right|}^{v}},-\sum_{j\in L_{r}^{v}}a_{j}\right)\right)^{\sim}\right]^{{\rm vir}}\right)
\end{align}
vanishes unless if each $J_{v}$ is non empty and it is equal to the union of at least one of the sets $I_{1},\dots,I_{k+1}$. In particular, this class vanishes whenever there is an unstable component $v\in V_{0}\left(\Phi\right)$ or if $k<\left|V_{0}\left(\Phi\right)\right|-1$.

Suppose first that $k=\left|V_{0}\left(\Phi\right)\right|-1$. In that case, each $J_{v}$ equals a unique $I_{j}$. For each choice of a pairing between the $k+1$ sets $J_{v}$ and the $k+1$ sets $I_{1},\dots,I_{k+1}$, the class in Eq.~(\ref{eq: push forward divisor}) gives 
\[
\prod_{v\in V_{0}\left(\Phi\right)}\left(\sum_{j\in L_{r}^{v}}a_{j}\right){\rm DR}_{g\left(v\right)}\left(a_{i_{1}^{v}},\dots,a_{i_{\left|L_{r}^{v}\right|}^{v}},-\sum_{j\in L_{r}^{v}}a_{j}\right).
\]
Now, adding up these contributions while taking care of the combinatorial factor appearing in Eq.~(\ref{eq:psi_infty}) proves the statement when $k=\left|V_{0}\left(\Phi\right)\right|-1$.

When $k>\left|V_{0}\left(\Phi\right)\right|-1$, there is at least one set $J_{v}$, for $v\in V_{0}\left(\Phi\right)$, which is the union of more than one set $I_{j}$. Once pushed forward by $s$, such contribution does not directly provide a single DR-cycle, but a sum of several DR-cyles attached to one another possibly by multiple edges, times a factor. Adding up these contributions by means of \cite[Theorem 4]{BSSZ} proves the last part of the statement.
\end{proof}

\subsubsection{Final step.}

It suffices now to write explicitly the contribution of 
\[
s_{*}\left(\frac{\prod_{i=n+1}^{n+m}{\rm ev}_{i}^{*}\left(\left[\infty\right]\right)}{\left|{\rm Aut}\Phi\right|\prod d_{e}}\iota_{*}\left(\frac{\left[\overline{\mathcal{M}}_{\Phi}\right]^{{\rm vir}}}{e_{\mathbb{C}^{*}}\left(\mathcal{N}_{\Phi}^{{\rm vir}}\right)}\right)\right)\lambda_{g}
\]
for each of the $3$ types of trees $\Phi$ described above. On these fixed locus, we have ${\rm ev}_{i}^{*}\left(\left[\infty\right]\right)=-au$.  If $v$ is a stable vertex over $0$, the push forward by $s$ is performed by Lemma~\ref{lem:Psi infinity}, and if $v$ is  a vertex over $\infty$, the push forward by $s$ forgets the spin structure an thus the image of the class $\tilde{\Omega}$ is the class $\Omega$. Adding up these contributions results in 

\begin{align}
\label{eq: loca with u}
& \delta_{m,1}a^{2}u^{2g-1}\sum_{d\geq0}\frac{\left(\bD_{g,n+1}(a_{1},\dots,a_{n})\right)_{d}}{u^{d}}
+\left(-1\right)^{m}a^{m+1}u^{2g-2+m}\sum_{d\geq0}\frac{\left(\bOm_{g,n}^{m}(a_{1},\dots,a_{n})\right)_{d}}{u^{d}}\\
\notag & +\left(-1\right)^{m}a^{m+1}u^{2g-2+m}\sum_{\substack{T\in\SRT(g,n,m)\\
\ell\in\mathcal{L}(t),\,\ell(T)=1
}
}\biggl(\prod_{e\in E(T)}\frac{a(e)}{u}\biggr)(b_{T})_{*}\bigg(\sum_{d_{1}\geq0}\frac{\left(\bOm(v_{r})\right)_{d_{1}}}{u^{d_{1}}}\otimes\bigotimes_{v \in V_{nr}(T)}\sum_{d_{2}\geq0}\frac{\left(\bD(v)\right)_{d_{2}}}{u^{d_{2}}}\bigg).
\end{align}
Extracting the coefficient of $u^{-d}$ for $d>0$ and simplifying by $(-1)^ma^{m+1}$ yields the relations of Theorem~\ref{thm:localization} after noticing that for each class in Eq.~(\ref{eq: loca with u}) the cohomological degree equals the polynomial degree.

\section{Proof of the main theorem}\label{sec:mainthm}

We first derive a linear relation between leveled trees decorated by classes $\bOmT$, $\bD$, and $\bUp$, and then use its basic properties to proof Theorem~\ref{thm:Omega-A}. The linear relation is of independent interest, as it doesn't use the specifics of the class $\bOm$ and an alternative version using $\psi$-classes is used in~\cite{BLS} in a different context, cf.~Remark~\ref{rem:Psi} below. 

\subsection{A linear relation between classes} If $T\in \SRT_{g,n,m}$, $v\in V_{nr}(T)$, then
\begin{align}
	\bUp(v)& \coloneqq \bUp^{|H_-(v)|}_{g(v),|H_+(v)|}(a(h_1),\dots,a(h_{|H_+(v)|}))
	\in R^*(\oM_{g(v),|H(v)|})\otimes_{\QQ}Q. 
\end{align}
On the other hand, for $v_r$ we shall use the notation 
\begin{align}
	\OmA(v_r) & \coloneqq \bOmT^m_{g,|H_+(v_r)|}(a(h_1),\dots,a(h_{|H_+(v_r)|})) -  \delta_{m,1}A^1_{g,|H_+(v_r)|}(a(h_1),\dots,a(h_{|H_+(v_r)|}))
\end{align}
Note that for $v_r$ we have $H_-(v_r)=m$. Consider the following sum of classes:
\begin{align} \label{eq:RelationSigmaDefinition}
	& \Sigma^m_{g,n}  \coloneqq - 
	(\bOmT^m_{g,n} - \delta_{m,1}A_{g,n}^{1} ) 
	\\
	\notag & \quad - \sum_{\substack{T\in\SRT(g,n,m)\\ d\in \mathcal{D}(T),\, \ell\in \mathcal{L}(T)\\ \ell(T)=1 \\ d^{\mathrm{lvl}} (0)  > 2g^{\mathrm{lvl}} (0)-2+m}} \biggl(\prod_{e\in E(T)} a(e)\biggr) (b_T)_* \bigg( (\OmA(v_r))_{d(v_r)}\otimes \bigotimes_{v \in V_{nr}(T)} (\bD(v))_{d(v)} \bigg) 
	\\ \notag
	& \quad + \hspace{-5mm}\sum_{\substack{T\in\SRT(g,n,m)\\  d\in \mathcal{D}(T),\, \ell\in \mathcal{L}(T)\\ \forall i < \ell(T)\colon  d^{\mathrm{lvl}} (i)\leq 2g^{\mathrm{lvl}} (i)-2+m}} 
	\hspace{-5mm}
	(-1)^{\ell(T)} \biggl(\prod_{e\in E(T)} a(e)\biggr) (b_T)_* \bigg(\bigotimes_{\substack{v \in V(T)\\ \ell(v)<\ell(T) }} (\bOm(v))_{d(v)} \otimes \bigotimes_{\substack{v \in V(T)\\ \ell(v)=\ell(T) }} (\bUp(v))_{d(v)} \bigg)
\end{align}

\begin{theorem} \label{thm:linear-relation} We have $\Sigma^m_{g,n}=0$. 
\end{theorem}

\begin{proof} Expand all graphs entering the definitions of $\bOmT^m_{g,n}$ in the first summand, $\OmA(v_r)$ in the second summand, and $\bUp(v)$ in the third summand. We have:
\begin{itemize}
	\item In the first term we replace 
	\begin{align}
		-\bOmT^m_{g,n}
	\end{align}
	by 
	\begin{align}\label{eq:Relation:SoloOmega-negative}
		\sum_{\substack{T\in\SRT(g,n,m)\\  d\in \mathcal{D}(T),\, \ell\in \mathcal{L}(T)\\ \forall i < \ell(T)\colon d^{\mathrm{lvl}} (i)\leq 2g^{\mathrm{lvl}} (i)-2+m 
	}} 
	\hspace{-5mm}
	(-1)^{\ell(T)+1} \biggl(\prod_{e\in E(T)} a(e)\biggr) (b_T)_* \bigotimes_{\substack{v \in V(T) }} (\bOm(v))_{d(v)} .
	\end{align}
	and the terms with $A^1_{g,n}$ we consider below in a separate computation.
	\item In the second summand we replace 
	\begin{align}
		- \sum_{\substack{T\in\SRT(g,n,m)\\ d\in \mathcal{D}(T),\, \ell\in \mathcal{L}(T)\\ \ell(T)=1 \\ d^{\mathrm{lvl}} (0) > 2g^{\mathrm{lvl}} (0)-2+m}} \biggl(\prod_{e\in E(T)} a(e)\biggr) (b_T)_* \bigg( (\bOmT(v_r))_{d(v_r)}\otimes \bigotimes_{v \in V_{nr}(T)} (\bD(v))_{d(v)} \bigg) 
	\end{align}
	by 
	\begin{align} \label{eq:Relation:OmegaD-bigger}
	\sum_{\substack{T\in\SRT(g,n,m)\\ d\in \mathcal{D}(T),\, \ell\in \mathcal{L}(T)\\ \ell(T) \geq 1 \\ \forall i < \ell(T)-1 \colon d^{\mathrm{lvl}} (i)\leq 2g^{\mathrm{lvl}} (i)-2+m \\ d^{\mathrm{lvl}} (\ell(T)-1) > 2g^{\mathrm{lvl}} (\ell(T)-1)-2+m }} 
	\hspace{-5mm}
	(-1)^{\ell(T)} \biggl(\prod_{e\in E(T)} a(e)\biggr) (b_T)_* \bigg( \bigotimes_{\substack{v \in V(T)\\ \ell(v)<\ell(T) }} (\bOm(v))_{d(v)} \otimes \bigotimes_{\substack{v \in V(T)\\ \ell(v)=\ell(T) }} (\bD(v))_{d(v)} \bigg).
	\end{align}
	The terms containing $A^1_{g(v_1,|H_+(v_r)|}$ we consider below in a separate computation.
	\item Finally, in the third summand, we expand the definition of $\bUp(v)$, cf.~\eqref{eq:Definition-Upsilon}. There is an exceptional case when $\ell(T)=0$ and $m=1$; we obtain the term  $\bD_{g,n+1}(a_1,\dots,a_{n})$ that we consider below in a separate computation, together with the exceptional terms from the previous summands. Removing this exceptional term, we have:
	\begin{align} 
		& \sum_{\substack{T\in\SRT(g,n,m)\\  d\in \mathcal{D}(T),\, \ell\in \mathcal{L}(T)\\ \forall i < \ell(T)\colon d^{\mathrm{lvl}} (i)\leq 2g^{\mathrm{lvl}} (i)-2+m}} 
		\hspace{-5mm}
		(-1)^{\ell(T))} \biggl(\prod_{e\in E(T)} a(e)\biggr) (b_T)_* \bigg(\bigotimes_{\substack{v \in V(T)\\ \ell(v)<\ell(T) }} (\bOm(v))_{d(v)} \otimes \bigotimes_{\substack{v \in V(T)\\ \ell(v)=\ell(T) }} (\bUp(v))_{d(v)} \bigg)
		\\ \notag & \qquad 
		-\bD_{g,n+1}(a_1,\dots,a_{n}) 
	\end{align}
	is equal to the sum of the following three terms:
	\begin{align} \label{eq:Relation:OmegaD-lesseq}
	\sum_{\substack{T\in\SRT(g,n,m)\\ d\in \mathcal{D}(T),\, \ell\in \mathcal{L}(T)\\ \ell(T) \geq 1 \\ \forall i < \ell(T) \colon d^{\mathrm{lvl}} (i)\leq 2g^{\mathrm{lvl}} (i)-2+m}} 
	\hspace{-5mm}
	(-1)^{\ell(T)} \biggl(\prod_{e\in E(T)} a(e)\biggr) (b_T)_* \bigg( \bigotimes_{\substack{v \in V(T)\\ \ell(v)<\ell(T) }} (\bOm(v))_{d(v)} \otimes \bigotimes_{\substack{v \in V(T)\\ \ell(v)=\ell(T) }} (\bD(v))_{d(v)} \bigg)
	\end{align}
	(this term occurs when we choose the $\bD$-class in the expansion of each $\bUp$ according to its definition),
	\begin{align} \label{eq:Relation-OmegaD-extralevel}
		\sum_{\substack{T\in\SRT(g,n,m)\\ d\in \mathcal{D}(T),\, \ell\in \mathcal{L}(T)\\ \ell(T) \geq 1 \\ \forall i < \ell(T)-1 \colon d^{\mathrm{lvl}} (i)\leq 2g^{\mathrm{lvl}} (i)-2+m  }} 
		\hspace{-5mm}
		(-1)^{\ell(T)+1} \biggl(\prod_{e\in E(T)} a(e)\biggr) (b_T)_* \bigg( \bigotimes_{\substack{v \in V(T)\\ \ell(v)<\ell(T) }} (\bOm(v))_{d(v)} \otimes \bigotimes_{\substack{v \in V(T)\\ \ell(v)=\ell(T) }} (\bD(v))_{d(v)} \bigg)
	\end{align}
		(this term occurs when we choose for at least one of $\bUp(v)$ the term supported on the leveled trees with $\ell(T)=1$, and we put the vertices with the $\bOm$-terms on the level $\ell(T)-1$ and the vertices with the $\bD$-terms on the level $\ell(T)$ for the expansions of all $\bUp(v)$), and
	\begin{align} \label{eq:Relation-AllOmegas}
		\sum_{\substack{T\in\SRT(g,n,m)\\  d\in \mathcal{D}(T),\, \ell\in \mathcal{L}(T)\\ \forall i < \ell(T)\colon d^{\mathrm{lvl}} (i)\leq 2g^{\mathrm{lvl}} (i)-2+m}}
		\hspace{-5mm}
		(-1)^{\ell(T)} \biggl(\prod_{e\in E(T)} a(e)\biggr) (b_T)_* \bigotimes_{\substack{v \in V(T) }} (\bOm(v))_{d(v)} 
	\end{align}
	(this term occurs when we choose replace all $\bUp(v)$ with the $\bOm$-terms in their definitions).
\end{itemize}
	
Now note that~\eqref{eq:Relation:OmegaD-bigger} and~\eqref{eq:Relation:OmegaD-lesseq} together are equal to~\eqref{eq:Relation-OmegaD-extralevel} with the opposite sign, and~\eqref{eq:Relation:SoloOmega-negative} is equal to~\eqref{eq:Relation-AllOmegas} with the opposite sign. Thus we see that in the expansion each leveled tree emerges exactly twice, with the opposite signs, except for the following exceptional terms that occur only for $m=1$ and that we left above for a separate computation:
\begin{align}
	& A_{g,n}^{1}  +  \bD_{g,n+1}(a_1,\dots,a_{n}) +
	\\ \notag 
	& \sum_{\substack{T\in\SRT(g,n,m)\\ \ell\in \mathcal{L}(T),\, \ell(T)=1}} \biggl(\prod_{e\in E(T)} a(e)\biggr) (b_T)_* \bigg( A^1_{g(v_r),|H_+(v_r)|}(a(h_1),\dots,a(h_{|H_+(v_r)|}))\otimes \bigotimes_{v \in V_{nr}(T)} \bD(v) \bigg).
\end{align}
One can show using~\cite[Lemmas~2.4 and 2.6]{BSSZ} and Equation~\eqref{eq:psi_infty} that the latter expression is equal to
\begin{align} 
	s_* \left(\lambda_g (t^*X) \left[\oM_{g}^{\sim}\left(\mathbb{P}^1,a_1,\dots,a_n,-\sum\nolimits_{i=1}^n a_i\right)\right]^{\mathrm{vir}}\right),
\end{align}
where we recall that $t\colon \oM_{g}^{\sim}(\mathbb{P}^1,a_1,\dots,a_n,-\sum_{i=1}^n a_i)\to LM_{2g-1+n}$ is the projection to the target curve. Here $X$ is the following class in the Losev-Manin space
\begin{align}
X = \frac{1}{1-\psi_0} - \frac{1}{1+\psi_\infty} +\sum_{D} \sum_{i,j} (-1)^{j+1} (b_D)_* \psi_{\infty'}^i \psi_{0'}^j \in A^*(LM_{2g-1+n}),	
\end{align}
 where the sum $\sum_{D}$ runs over all divisors in the Losev-Manin space, $b_D$ is the boundary pushforward map corresponding to $D$, and $0'$ and $\infty'$ are the two branches at the node. Since $X=0$ (this follows from $(X)_1=0$ by induction), we obtain the vanishing stated in the theorem in the $m=1$ case as well.
\end{proof}

Now, several remarks on relation~\eqref{eq:RelationSigmaDefinition} are in order. 

\begin{remark} \label{rem:sigma-tilde}
	All terms in $\Sigma^m_{g,n}$ are supported on graphs with at least one edge, with the following exceptions: $- 
	(\bOmT^m_{g,n} - \delta_{m,1}A_{g,n}^{1} ) $ (which is the first summand in~\eqref{eq:RelationSigmaDefinition}) and $\bUp^m_{g,n}$ (which corresponds to the only graph with one vertex in the third summand in~\eqref{eq:RelationSigmaDefinition}). Thus, 
\begin{align}
	\Sigma^m_{g,n} = \bUp^m_{g,n}- (\bOmT^m_{g,n} - \delta_{m,1}A_{g,n}^{1} ) + \tilde\Sigma^m_{g,n},
\end{align}
where $\tilde \Sigma^m_{g,n}$ is represented by graphs with at least one edge.
\end{remark}

\begin{remark} \label{rem:dimension-Ups}
	Let $d\geq 2g-1+m$. Then note that each graph in the $(\tilde \Sigma^m_{g,n})_d$ either contains the rood vertex labeled by $(\bOmT^m_{g',n'} -  \delta_{m,1}A^1_{g',n'})_{d'}$ with $d'\geq 2g'-1+m$, $2g'-2+n' \leq 2g-2+n$ (these graphs are in the second line of~\eqref{eq:RelationSigmaDefinition}) or at least one non-root vertex labeled by $(\bUp^1_{g'n'})_{d'}$ with $d'\geq 2g'$ (these graphs are the ones with at least one edge in the third line of~\eqref{eq:RelationSigmaDefinition}). In order to see the latter property, note that since for each graph with at least one edge in the third line of~\eqref{eq:RelationSigmaDefinition}, we have $d^{\mathrm{lvl}}(\ell(T)-1) \leq 2g^{\mathrm{lvl}}(\ell(T)-1) -2+m$. Now, since
\begin{align}
	d & = d^{\mathrm{lvl}}(\ell(T)-1) + \sum_{\substack{v\in V(T)\\ \ell(v)=\ell(T)}} (d(v) +1), \\ \notag
	g & = g^{\mathrm{lvl}}(\ell(T)-1)+ \sum_{\substack{v\in V(T)\\ \ell(v)=\ell(T)}} g(v),
\end{align}
the condition $d\geq 2g-1+m$ implies that there exists at least one vertex $v\in V(T)$ such that $\ell(v)=\ell(T)$ satisfying $d(v)\geq 2g(v)$. 
\end{remark}

\begin{remark} [used in~\cite{BLS}] \label{rem:Psi} Note that the proof of Theorem~\ref{thm:linear-relation} and the remarks above don't use any specifics of the class $\bOm^{m}_{g,n}(a_1,\dots,a_n)$. We can use any other class that belongs to the ring $R^*(\oM_{g,n+m})\otimes_{\QQ}Q$ and whose part of cohomological degree $d$ is a homogeneous polynomial of degree $d$ in $a_1,\dots,a_n$. For instance, one can replace $\bOm^{m}_{g,n}(a_1,\dots,a_n)$ everywhere (that is, all its explicit occurences in~\eqref{eq:RelationSigmaDefinition} as well as in the definitions of $\bOmT$, $\bUp$ and $\OmA$) by $\Psi^m_{g,n} (a_1,\dots,a_n)\coloneqq \prod_{i=1}^n (1-a_i\psi_i)^{-1}\in R^*(\oM_{g,n+m})\otimes_{\QQ}Q$. 
\end{remark}

\subsection{Proof of Theorem~\ref{thm:Omega-A}} Now Theorem~\ref{thm:Omega-A} becomes an easy corollary of Theorems~\ref{thm:localization} and~\ref{thm:linear-relation}. We prove it by induction, taking into account Remarks~\ref{rem:sigma-tilde} and~\ref{rem:dimension-Ups}. 

Indeed, Remark~\ref{rem:sigma-tilde} along with Theorem~\ref{thm:linear-relation} implies that the vanishing of $(\bOmT^m_{g,n} - \delta_{m,1}A_{g,n}^{1} )_d$, $d\geq 2g-1+m$, is reduced to the vanishing of $(\bUp^m_{g,n})_d$ established by Theorem~\ref{thm:localization}, and the vanishing of $(\tilde \Sigma^m_{g,n})_d$. Now, by Remark~\ref{rem:dimension-Ups}, each graph that enters $(\tilde \Sigma^m_{g,n})_d$ for $d\geq 2g-1+m$ either contains a vertex labeled by $(\bOmT^m_{g',n'} -  \delta_{m,1}A^1_{g',n'})_{d'}$ with $d'\geq 2g'-1+m$, $2g'-2+n' \leq 2g-2+n$, which vanished by the induction assumption, or a vertex labeled by $(\bUp^1_{g'n'})_{d'}$ with $d'\geq 2g'$, which vanishes by Theorem~\ref{thm:localization}. 

This completes the proof of Theorem~\ref{thm:Omega-A}.

\end{document}